\documentclass[a4paper,11pt]{article}
\usepackage{amssymb,amsmath,amsfonts,amsthm}
\newcommand\dslash{d\llap {\raisebox{.9ex}{$\scriptstyle-\!$}}}
\newtheorem{theorem}{Theorem}[section]

\newtheorem{lemma}[theorem]{Lemma}
\newtheorem{proposition}[theorem]{Proposition}

\newtheorem{remark}{Remark}

\newcommand{\R}{\mathbb{R}}

\def\N{\mathbb N}
\def\C{\mathbb C}

\def\ta{\tilde{\alpha}}
\def\tb{\tilde{\beta}}

\def\tg{\tilde{\gamma}}

\def\ve{\varepsilon}

\def\peso{(2\pi)^{-n}}

\def\pxi{\langle \xi \rangle}

\def\px{\langle x \rangle}
\def\pd{\langle D \rangle}

\newcommand{\veps}{\varepsilon}

\newcommand{\db}{\;\rule[2.3mm]{1.9mm}{.2mm}\!\!\!d}

\begin{document}
\catcode`\@=11


  \renewcommand{\theequation}{\thesection.\arabic{equation}}
  \renewcommand{\section}%
  {\setcounter{equation}{0}\@startsection {section}{1}{\z@}{-3.5ex plus -1ex
   minus -.2ex}{2.3ex plus .2ex}{\Large\bf}}
\title{\bf Decay estimates for solutions \\ of nonlocal semilinear equations}
\author{Marco Cappiello $^{\textrm{a}}$, Todor Gramchev $^{\textrm{b}}$ and Luigi Rodino $^{\textrm{c}}$ \\ \\
}
\date{}
\maketitle

\begin{abstract}
We investigate the decay for $|x|\rightarrow \infty$ of weak Sobolev type solutions of semilinear nonlocal equations $Pu=F(u)$. We consider the case when $P=p(D)$ is an elliptic Fourier multiplier with polyhomogeneous symbol $p(\xi)$ and derive sharp algebraic decay estimates in terms of weighted Sobolev norms. In particular, we state a precise relation between the singularity of the symbol at the origin and the rate of decay of the corresponding solutions. Our basic example is the celebrated Benjamin-Ono equation \begin{equation} \label{BO}(|D|+c)u=u^2, \qquad c>0,\end{equation} for internal solitary waves of deep stratified fluids. Their profile presents algebraic decay, in strong contrast with the exponential decay for KdV shallow water waves.
\end{abstract}

\section{Introduction}
The main goal of the present paper is to investigate the appearance of algebraic decay at infinity for weak solutions of semilinear nonlocal elliptic equations of the form 
\begin{equation} \label{equation}
Pu= F(u),
\end{equation}
where $P=p(D)$ is a Fourier multiplier in $\R^n$:
\begin{equation}
\label{FM}
Pu(x)= \int_{\R^{n}}e^{ix\xi}p(\xi) \hat{u}(\xi) \dslash \xi,
\end{equation}
with $\hat{u}(\xi)= \int_{\R^{n}}e^{-ix\xi}u(x)dx$, $\dslash \xi = \peso d\xi$ and $F(u)$ is a polynomial vanishing of order $k\geq 2$ at $u=0$, namely 
\begin{equation} \label{nonl1}
F(u) = \sum_{j=2}^N F_j u^j, \qquad F_j \in \C.
\end{equation}
For $P$ operator with constant coefficients, i.e. $p(\xi)$ polynomial, or more general Fourier multiplier, equations of the form \eqref{equation} arise frequently in Mathematical Physics in the theory of solitary waves for nonlinear evolution equations.
Relevant examples are equations in the realm of wave motions featuring both dispersion and diffusion processes, long internal waves and the interface between two fluids of different densities, and semilinear Schr\"odinger equations. Let us recall in short, in the case $x \in \R$: starting from an evolution equation of the form 
$v_t  +(Pv)_x = F(v)_x,$ with $t \geq 0,$ solitary waves are solutions of the form $v(t,x)=u(x-ct), c>0$. Looking for this type of solutions one is indeed reduced to study the elliptic equation \eqref{equation}. \\ There are no general methods for deriving the existence of such special solutions and in the known examples the special features like conservation laws and/or the presence of symmetries play a fundamental role. On the other hand, it is natural to study regularity and behaviour at infinity of this type of waves in order to have a global knowledge of their profile.
In the fundamental papers \cite{BL1, BL2}, Bona and Li proved that if $p(\xi)$ is analytic on $\R$, then every solution $u \in L^{\infty}(\R)$ of \eqref{equation} such that $u(x) \rightarrow 0$ for $x \rightarrow \pm \infty$, exhibits an exponential decay of the form $e^{-\veps|x|}, \veps>0$ for $|x|\rightarrow \infty$ and extends to a holomorphic function in a strip of the form $\{z \in \C: |\Im z|<T\}$ for some $T>0.$ The researches in \cite{BL1, BL2} were motivated by the applications to the study of decay and analyticity of solitary waves for KdV-type, long-wave-type and Schr\"odinger-type equations. In \cite{BG} and \cite{CGR2}, the results of \cite{BL2} have been extended in arbitrary dimension to analytic pseudodifferential operators, deriving sharp estimates in the frame of the Gelfand-Shilov spaces of type $\mathcal{S}$, cf. \cite{GS2}, which give a simultaneous information on the decay at infinity and the Gevrey-analytic regularity on $\R^n$.  Recently, the results on the holomorphic extensions have been refined in \cite{CN1}, \cite{CN2}. \\ Here we want to consider the case when $p(\xi)$ is only finitely smooth at $\xi=0.$ Namely the symbol $p(\xi)$ is assumed to be a sum of positively homogeneous terms and we are interested in the nonlocal case, i.e. at least one of  these terms is not a polynomial, hence $p(\xi)$ is only finitely smooth at the origin. In this case, the functional analytic machinery and the pseudodifferential calculus used in the above mentioned papers are not applicable. Motivation for this type of study comes from two directions. The first is the presence of several nonlinear models in the theory of solitary waves in which the symbol of the linear part presents singularities or finite smoothness at $\xi=0.$
The most celebrated equation in this category is the so called Benjamin-Ono equation \eqref{BO}, cf. \cite{AT, B1, IKK, O}, which will be considered in detail in the next Section \ref{examples}. \\ Another more general issue comes from the novelty with respect to the general theory on decay and regularity estimates for linear and nonlinear elliptic equations in $\R^{n}$; besides \cite{BG, BL1, BL2,  CGR2, CN1, CN2}, see for example \cite{AG, CGR3, Co, Mc, McL, McL2, Ra}. In fact, we are not aware of any result for general semilinear elliptic equations of the form \eqref{equation} in the case of finitely smooth symbols. As a first step, in this paper we will focus on the decay of the solutions with the purpose of treating analytic regularity and holomorphic extensions in a future work. With respect to the case of smooth or analytic symbols, a finite smoothness of the symbol of the linear part of \eqref{equation} determines the loss of the rapid or exponential decay observed in all the above mentioned papers and the appearance of an algebraic decay at infinity. We want to state a precise relation between the regularity of the symbol $p(\xi)$ and the rate of decay at infinity of the solutions. This will be given, as a natural choice, in terms of estimates in the weighted Sobolev spaces $$H^{s,t}(\R^{n}):= \{u \in \mathcal{S}'(\R^{n}): \|u\|_{s,t}=\|\px^{t}\pd^{s}u\|_{L^{2}(\R^n)}<\infty \}, \quad s,t \in \R,$$
where $\px=(1+|x|^{2})^{1/2}$ and $\pd^{s}$ denotes the multiplier with symbol $\pxi^{s}.$ Notice that for $t=0, H^{s,0}(\R^{n})$  coincides with the standard Sobolev space $H^{s}(\R^{n}).$ In the sequel we shall denote as standard by $\|\cdot \|_s$ the norm $\|\cdot\|_{s,0}.$ We refer the reader to \cite{Co} for a detailed presentation of the properties of these spaces. \\ Let us now detail the class of operators $p(D)$ to which our results apply. 
We shall consider Fourier multipliers with symbols of the following type 
\begin{equation} \label{linop1}
p(\xi) = p_0+\sum_{j=1}^h  p_{m_j} (\xi)
\end{equation}
where $p_0 \in \C$ and $p_{m_j}(\xi) \in C^{\infty}(\R^{n}\setminus 0)$ are (positively) homogeneous symbols of order $m_j,$ i.e. $p_{m_{j}}(\lambda \xi)=\lambda^{m_{j}}p(\xi)$ for $\lambda >0$, with $0<m_1<m_2<\ldots<m_h=M$.
We assume that $M \geq 1$ and that the following global ellipticity condition holds: 
\begin{equation} 
\inf_{\xi\in \R^n}(\pxi^{-M}|p(\xi)|)  > 0. \label{ell} \end{equation}
Since $p(0)=p_0$, condition \eqref{ell} implies in particular that $p_0\neq 0$.\\
Set moreover
\begin{equation} \label{linop1b}
m:=\min\{m_j : p_{m_j}  \textrm{ is not polynomial }  \}.
\end{equation}
We shall call $m$ the \textit{singularity index} of $p(\xi)$. When the set in the right-hand side of \eqref{linop1b} is empty, then $P=p(D)$ is a partial differential operator with constant coefficients and we go back to the above mentioned results of exponential decay. \\

Our main result is the following.

\begin{theorem}\label{thm1}
Let $m \in \R$ with  $ [m]>n/2$ and let $P$ be an operator with symbol $p(\xi)$ of the form \eqref{linop1}, \eqref{linop1b} satisfying the assumption \eqref{ell}. Assume that $u$ is a solution of \eqref{equation} such that $u \in H^{s,\ve_o}(\R^n)$ for some $s>n/2$ and for some $\veps_o >0$. 
Then, $u \in C^{\infty}(\R^n)$ and for every $\alpha \in \N^n$ and $\veps>0$ we have $$\partial^{\alpha}u \in H^{s,|\alpha|+m+n/2-\veps}(\R^n),$$ i.e. the following estimate holds
\begin{equation}\label{decay}\| \px^{m+n/2-\veps} x^{\beta} \partial^{\alpha}u\|_{s} <\infty \end{equation}
for every $\alpha, \beta \in \N^n$, with $|\beta|\leq |\alpha|$. Under the same assumptions on $p(D)$ and $u$ the same result holds for solutions of the equation \begin{equation}
\label{fequation} p(D)u=f+F(u),\end{equation} where $f$ is a given smooth function satisfying \eqref{decay}.
\end{theorem}

Since \eqref{decay} implies 
\begin{equation}\label{decay2}
\| \px^{m+n/2-\veps}u\|_{L^{2}}< \infty,
\end{equation}
arguing roughly in terms of algebraic decay, we expect $u(x)=O(|x|^{-m-n})$ for $|x| \rightarrow \infty.$ This corresponds indeed to the behaviour of the solutions of the Benjamin-Ono equation and some similar equations which we shall solve explicitly in Section \ref{examples}. Note however that general results of existence and uniqueness are out of reach of our methods, addressed only to the qualitative behaviour of the solutions, see Sections 3 and 4 for the proof of Theorem \ref{thm1}. \\ Finally, we observe that the second part of Theorem \ref{thm1} turns out to be new also for linear equations, i.e. when $F(u)=0$ in \eqref{fequation}, whereas the first part is trivial in this case since the homogeneous equation $p(D)u=0$ admits only the solution $u=0$, the nonlinearity being essential to produce non-trivial solutions when $f=0$. 


\begin{section}{Examples} \label{examples}
This section is devoted to the analysis of some examples and models satisfying the assumptions of Theorem \ref{thm1}. In particular we shall test on these models the sharpness of our decay estimates. The first and more important example is given by the Benjamin-Ono equation in hydrodynamics.\\

\noindent
\textbf{Benjamin-Ono equation.} The Benjamin-Ono equation was introduced by Benjamin \cite{B1} and Ono \cite{O} and describes one-dimensional internal waves in stratified fluids of great depth. It reads as follows:
\begin{equation}\label{tBO}
\partial_t v + H(\partial_x^2v)+2v\partial_x v=0, \qquad t \in \R, x \in \R,
\end{equation}
where $H(D)$ stands for the Hilbert transform, i.e. the Fourier multiplier operator of order $0$ with symbol $-i \textrm{sign}\xi:$
\begin{equation} \label{hilbert}
H(D)u(x)= \frac{1}{\pi} \textrm{P.V.} \int_{\R} \frac{u(y)}{x-y} \, dy = \int_{\R} e^{ix\xi}(-i \textrm{sign}\xi)\hat{u}(\xi) \dslash \xi.
\end{equation}
There exists a large number of papers dealing with existence, uniqueness and time asymptotics for the initial value problem related to the equation \eqref{tBO} and its generalizations in various functional settings, see for instance \cite{BL, BP, FP, LPP, MST, RS, T}. Concerning the solitary waves $u(x-ct), c>0,$ they satisfy the nonlocal elliptic equation \eqref{BO}
which corresponds to \eqref{equation} for $p(\xi)=|\xi|+c$ and $F(u)=u^2.$
In \cite{B1} Benjamin found the solution \begin{equation}
\label{solBO}
u(x)=\frac{2c}{1+c^2x^2}, \quad x \in \R,
\end{equation}
(see \eqref{2.15} in Remark \ref{rem2.3} for the easy computation). Later, Amick and Toland \cite{AT} proved that, apart from translations, the function \eqref{solBO} is the only solution of \eqref{BO} which tends to $0$ for $|x| \rightarrow \infty.$ Notice that $u(x)$ in \eqref{solBO} exhibits a quadratic decay at infinity like $|x|^{-2},$ satisfying \eqref{decay2} with $m=1, n=1.$\\

In the following we give a similar example with $m=1, n=1$,  which is not related to applicative problems. It confirms the optimality of the results stated in Theorem \ref{thm1} and testifies  that the polynomial terms $p_{m_j}(\xi)$ in the expression of $p(\xi)$ have no influence on the rate of decay.
 
\begin{theorem} \label{thmexamples} In dimension $n=1$, the equation \begin{equation}
\label{ex1}
D^2u+3|D|u+3u=8u^3
\end{equation}
admits the solution 
\begin{equation} \label{solex1}
u(x)=\frac{1}{1+x^2}, \qquad x \in \R.
\end{equation}
(In \eqref{ex1} we mean $D^{2}u=-u''$ and $|D|$ is as before the Fourier multiplier with symbol $|\xi|, \xi \in \R$).
\end{theorem}

Note that the linear part of \eqref{ex1} is globally elliptic, i.e. \eqref{ell} is satisfied. Moreover, the order is $M=2$ and the singularity index is $m=1.$

\begin{proof}
We shall check that the Fourier transforms of the left and right-hand sides of \eqref{ex1} coincide for $u(x)$ as in \eqref{solex1}.
To this end we recall (see for example \cite{Sc}, formula (VII, 7;23), page 260, or \cite{GS}, formula (9), page 187) that
\begin{equation}
\label{bessel}
\mathcal{F}((1+x^2)^{-\lambda})=\frac{2\sqrt{\pi}}{\Gamma(\lambda)}\left( \frac{|\xi|}{2}\right)^{\lambda-\frac{1}{2}}K_{\lambda-\frac{1}{2}}(|\xi|),
\end{equation}
where $x,\xi \in \R^n$ and $\Gamma$ denotes the standard Euler function; arguing in the distribution sense, we may allow any $\lambda>0$. The functions $K_{\nu}(x), \nu \in \R, x \in \R \setminus 0$, are the modified Bessel functions of second type; for definitions and properties see for example \cite{EM}, \cite{Wa}.
We recall in particular that
\begin{equation}
\label{prop1}
K_{\nu}(x)=K_{-\nu}(x), \qquad \nu \in \R, x \neq 0,
\end{equation}
\begin{equation}\label{prop2}
K_{\nu+1}(x)=\frac{2\nu}{x}K_{\nu}(x)+K_{\nu-1}(x), \qquad \nu \in \R, x \neq 0.
\end{equation}
From \eqref{prop1}, \eqref{prop2}, we have
\begin{equation}
\label{prop4}
K_{\frac{3}{2}}(x)=\left( \frac{1}{x}+1\right)K_{\frac{1}{2}}(x), \qquad x\neq 0,
\end{equation}
\begin{equation}
\label{prop5}
K_{\frac{5}{2}}(x)=\left( \frac{3}{x^2}+\frac{3}{x}+1\right)K_{\frac{1}{2}}(x), \qquad x\neq 0.
\end{equation}
Let us then prove that \eqref{solex1} is a solution of \eqref{ex1}. In fact, from \eqref{bessel} and \eqref{prop5} we have
\begin{multline*}
8\mathcal{F}(u^3)=8\mathcal{F}((1+x^2)^{-3})\\= 8\sqrt{\pi}\left(\frac{|\xi|}{2}\right)^{5/2}K_{\frac{5}{2}}(|\xi|) 
= 2\sqrt{\pi}(\xi^2+3|\xi|+3)\left(\frac{|\xi|}{2}\right)^{1/2}K_{\frac{1}{2}}(|\xi|) \\
=(\xi^2+3|\xi|+3)\mathcal{F}((1+x^2)^{-1})=\mathcal{F}(D^2u+3|D|u+3u). \end{multline*}
\end{proof}

\begin{remark}
\label{rem2.3} The method used in the proof of Theorem \ref{thmexamples} can also be applied to the Benjamin-Ono equation and allows to give an easy alternative proof that the function $u(x)$ in \eqref{solBO} is a solution of \eqref{BO}, say for $c=1.$ In fact, from \eqref{bessel}
and \eqref{prop4} we easily obtain
\begin{multline} \label{2.15}
\mathcal{F}(u^2)=4\mathcal{F}((1+x^2)^{-2})= 8\sqrt{\pi}\left(\frac{|\xi|}{2}\right)^{3/2}K_{\frac{3}{2}}(|\xi|) \\
=4\sqrt{\pi}(|\xi|+1)\left(\frac{|\xi|}{2}\right)^{1/2}K_{\frac{1}{2}}(x) = 2(|\xi|+1) \mathcal{F}((1+x^2)^{-1}) \\ = \mathcal{F}(|D|u+u).
\end{multline}
Note that by \eqref{prop2} we may calculate inductively $K_{N/2}(x)$ for any odd integer $N$ in terms of $K_{1/2}(x)$. This allows to produce other similar examples, with higher order $M$, with $m=1$ and higher order nonlinearity. Solutions are still of the form $u(x)=\frac{1}{1+x^2}.$
\end{remark}

\end{section}

\begin{section}{Commutator identities and estimates}
In this section we  prove some commutator identities for Fourier multipliers which will be used in the proof of our result.
We first state a simple but crucial assertion on the compensation of the singularities at $\xi=0$ for homogeneous symbols.

\begin{lemma}\label{homo}
Let $p(\xi)$ be of the form \eqref{linop1} satisfying \eqref{ell} and let $m$ be defined by \eqref{linop1b}. Then the following estimates hold:
\begin{equation}
\sup_{\xi\in \R^n}\frac{|D_\xi^{\sigma}(\xi^{\tg} D_\xi^\gamma p(\xi))|}{|p(\xi)|}  < +\infty, \quad \gamma, \tg,\sigma \in \N^n, |\gamma|= |\tg|,|\sigma|\leq [m]. \label{linop3}\end{equation}
\end{lemma}

\begin{proof}
Since $|\tg|=|\gamma|$, then $D_\xi^{\sigma}(\xi^{\tg} D_\xi^\gamma p(\xi))$ is a sum of terms with homogeneity of order $m_j-|\sigma|.$ Since $m_j \leq M$, we have $m_j-|\sigma|\leq M.$ Moreover, in the non-polynomial case, in view of the assumptions $|\sigma|\leq [m], m_j \geq m,$ we have $m_j-|\sigma| \geq m-[m] \geq 0.$ Therefore, for some $C>0$ we have
$$|D_\xi^{\sigma}(\xi^{\tg} D_\xi^\gamma p(\xi))| \leq C \pxi^M, \qquad \xi \in \R^n.$$ Hence \eqref{linop3} follows from \eqref{ell}.
\end{proof}

\begin{remark} \label{rem3.2}
In Lemma \ref{homo}, and often in the sequel, we consider higher order derivatives of the non-polynomial terms $D_{\xi}^{\gamma}p_{m_{j}}(\xi).$ These derivatives should be performed in the distribution sense, possibly producing $\delta$ distribution or its derivatives at the origin. However, in all the expressions, multiplication by monomials $\xi^{\alpha}$ appears as well, so that in the whole we shall always obtain a distribution $h \in \mathcal{S}'(\R^{n})$ homogeneous of order larger than $-n$. Then $\delta$ contributions are cancelled. Strictly speaking: the distribution $h \in \mathcal{S}'(\R^{n})$ can be identified in this case with the function $h_{|_{\R^{n} \setminus 0}} \in C^{\infty}(\R^{n}\setminus 0) \cap L^{1}_{\textrm{loc}}(\R^{n}).$ Let us refer, for example, to \cite{GS}, Chapter 1, Section 3.11, for a detailed explanation. Summing up, in Lemma \ref{homo} and in the sequel we may limit ourselves to argue in classical terms, i.e. on the pointwise definition of derivatives.
\end{remark}

\begin{proposition} \label{commid}
Let $p(D)$ be a Fourier multiplier defined by a homogeneous symbol $p(\xi)$ of order $m \geq 0$ and let $\alpha, \beta \in \N^n$ with $|\beta|\leq |\alpha|.$ Then, for every $u \in \mathcal{S}(\R^n)$ the following identity holds:  
\begin{multline}
x^{\beta}p(D)D^{\alpha}u  = p(D)(x^\beta  D^\alpha u ) \\+ \sum_{0\neq \gamma \leq \beta}\sum_{\stackrel{\ta, \tb}{|\tb|\leq|\ta|<|\alpha|}}C_{\alpha \beta \gamma \ta \tb \tg} D^{\tg}\circ (D_\xi^{\gamma} p)(D)(x^{\tb}D_x^{\ta}u),
\label{puffo}
\end{multline}
where for every $\gamma$ in the sums above, $\tg$ denotes a multi-index depending on $\alpha, \beta, \ta, \tb, \gamma$ and satisfying the condition $|\tg|=|\gamma|$, and $C_{\alpha \beta \gamma \ta \tb \tg}$ are suitable constants. 
\end{proposition}

\begin{proof}
We can write 
\begin{multline*}
x^{\beta}p(D)D^{\alpha}u = \sum_{\gamma \leq \beta} \binom{\beta}{\gamma} \int_{\R^n}\int_{\R^n}e^{i(x-y)\xi} (x-y)^{\gamma}p(\xi)y^{\beta-\gamma}D_y^{\alpha}u(y)dy \db \xi \\
= p(D)(x^\beta  D^\alpha u ) +  \sum_{0\neq \gamma \leq \beta} \binom{\beta}{\gamma} \int_{\R^n}\int_{\R^n}D_{\xi}^{\gamma}(e^{i(x-y)\xi}) p(\xi)y^{\beta-\gamma}D_y^{\alpha}u(y)dy \db \xi, \end{multline*}
Integration by parts with respect to $y$ and $\xi$ gives 
\begin{multline*} \int_{\R^n}\int_{\R^n}D_{\xi}^{\gamma}(e^{i(x-y)\xi}) p(\xi)y^{\beta-\gamma}D_y^{\alpha}u(y)dy \db \xi \\
= \int_{\R^n}\int_{\R^n}D_{\xi}^{\gamma}(-D_y)^{\alpha}(e^{i(x-y)\xi}y^{\beta-\gamma}) p(\xi)u(y)dy \db \xi \\
= \sum_{\stackrel{\delta \leq \alpha}{\delta\leq \beta-\gamma}} (-1)^{|\gamma|}(-i)^{|\delta|} \binom{\alpha}{\delta} \frac{(\beta-\gamma)!}{(\beta-\gamma-\delta)!}
\int_{\R^n}\int_{\R^n} e^{i(x-y)\xi}\xi^{\alpha-\delta}(D_{\xi}^{\gamma}p)(\xi)y^{\beta-\gamma-\delta}u(y)dy\db \xi.\end{multline*}
Let now $\tg$ be a multi-index such that $\tg \leq \alpha-\delta$ and $|\tg|=|\gamma|.$ Such a multi-index exists since $|\gamma| \leq |\beta-\delta|\leq |\alpha-\delta|$ in the sums above. Then, write $$e^{i(x-y)\xi}\xi^{\alpha-\delta}=\xi^{\tg}(-D_y)^{\alpha-\delta-\tg}e^{i(x-y)\xi}$$ and integrate by parts again with respect to $y$. We obtain
\begin{multline*}
 \int_{\R^n}\int_{\R^n} e^{i(x-y)\xi}\xi^{\alpha-\delta}(D_{\xi}^{\gamma}p)(\xi)y^{\beta-\gamma-\delta}u(y)dy\db \xi \\
=\int_{\R^n}\int_{\R^n} e^{i(x-y)\xi}\xi^{\tg}(D_{\xi}^{\gamma}p)(\xi)D_y^{\alpha-\delta-\tg}(y^{\beta-\gamma-\delta}u(y))dy\db \xi \\
=\sum_{\stackrel{\theta \leq \alpha-\delta-\tg}{\theta \leq \beta-\gamma-\delta}}(-i)^{|\theta|}\binom{\alpha-\delta-\tg}{\theta} \frac{(\beta-\gamma-\delta)!}{(\beta-\gamma-\delta-\theta)!} \times \\ \times \int_{\R^n}\int_{\R^n} e^{i(x-y)\xi}\xi^{\tg}(D_{\xi}^{\gamma}p)(\xi)y^{\beta-\gamma-\delta-\theta}D_y^{\alpha-\tg-\delta-\theta}u(y)dy\db \xi,
\end{multline*}
which gives \eqref{puffo}. 
\end{proof}

\begin{proposition}
\label{commid2} 
Let $q(D)$ be a Fourier multiplier defined by a homogeneous symbol $q(\xi)$ of order $m>0.$ Then, for every $\rho \in \N^n$ with $|\rho|<m+n$ and for every $v \in \mathcal{S}(\R^n)$ the following identity holds:
\begin{equation}
x^\rho q(D) v = q(D)(x^{\rho}v)+ 
\sum_{0\neq \sigma \leq \rho} \binom{\rho}{\sigma}(-1)^{|\sigma|}(D_{\xi}^\sigma q)(D)(x^{\rho-\sigma} v). \label{multcomm} \end{equation}
\end{proposition}

\begin{proof} Notice that the condition $|\sigma|<m+n$ and the homogeneity imply that $D_{\xi}^{\sigma}q(\xi) \in L^1_{\textrm{loc}}(\R^n)$. Then integrating by parts we have 
\begin{eqnarray*}
x^\rho q(D) v &=& \int_{\R^n} (D_{\xi}^{\rho}e^{ix\xi})q(\xi)\hat{v}(\xi)\db \xi \\ &=&
\sum_{\sigma \leq \rho}\binom{\rho}{\sigma}(-1)^{|\sigma|} \int_{\R^n}e^{ix\xi}D_{\xi}^{\sigma}q(\xi)D_{\xi}^{\rho-\sigma}\hat{v}(\xi)\db \xi
\end{eqnarray*}
from which \eqref{multcomm} follows.
\end{proof}

Fixed $s \in \R$, we shall denote by $H^s_1(\R^n)$ the space of all $u \in \mathcal{S}'(\R^n)$ such that $$\|u\|_{H^s_1}:= \|\pd^s u\|_{L^1}<\infty.$$
The next result states some useful estimates for singular operators, that is operators with symbol $q(\xi) \rightarrow \infty$ for $\xi \rightarrow 0.$

\begin{lemma} \label{homodist}
Let $q(\xi) \in C^{\infty}(\R^n\setminus 0)$ be a homogeneous symbol of order $\mu \in (-n/2,0)$ and let $\varphi \in C_o^{\infty}(\R^n)$ such that $\varphi(\xi)=1$ for $|\xi|\leq 1$. Consider the operator 
$$H_{\varphi, q}v(x):= ((\varphi q)(D)v)(x)= \int_{\R^n}\int_{\R^n}e^{i(x-y)\xi}\varphi(\xi)q(\xi)v(y)dy\db \xi, \qquad v \in \mathcal{S}(\R^n).$$ Then we have
\begin{equation}
\label{L2L1}
\|H_{\varphi, q}v\|_{H^s} \leq C_s \|v\|_{H^s_1}.
\end{equation}
\end{lemma}

\begin{proof}
Observe that we can write $\varphi(\xi)q(\xi) = q(\xi)-(1-\varphi(\xi))q(\xi).$ Since $q(\xi)$ is a homogeneous distribution of order $\mu,$ then its inverse Fourier transform is a homogeneous distribution of order $-n-\mu.$ On the other hand, it is immediate to check that the inverse Fourier transform of $(1-\varphi(\xi))q(\xi)$ is rapidly decreasing. Then we have that 
$$|\mathcal{F}^{-1}_{\xi\rightarrow x}(\varphi(\xi)q(\xi))(x)| \leq C\px^{-n-\mu}.$$ Since $\mu >-n/2$, the estimate above implies that $\mathcal{F}^{-1}_{\xi\rightarrow x}(\varphi(\xi)q(\xi))(x) \in L^2(\R^n).$ Hence, writing $$H_{\varphi, q}v(x)= (\mathcal{F}^{-1}_{\xi\rightarrow x}(\varphi(\xi)q(\xi))\ast v)(x),$$ the estimate \eqref{L2L1} follows as a consequence of Young inequality.
\end{proof}

We address now the case of commutation with fractional powers.

\begin{lemma}\label{fractional}
Let $q(\xi)$ be a smooth positively homogeneous symbol of order $\mu$, let $r >0$ and $\varphi \in C_o^{\infty}(\R^n)$ such that $\varphi(\xi)=1$ for $|\xi|\leq 1.$ Then, if $\mu-r >-n/2$ then for every $v \in \mathcal{S}(\R^n)$ we have:
\begin{equation}\label{sob1}\|[\px^r, H_{\varphi, q}]v\|_{s} \leq C_s \| v\|_{H_1^s}.\end{equation} If moreover $\mu-r >0$, then 
\begin{equation}\label{sob2}\|[\px^r, H_{\varphi, q}]v\|_s \leq C_s \|v\|_s.\end{equation} \end{lemma}
\begin{proof} 
Writing explicitly the commutator we have:
$$[\px^r, H_{\varphi, q}]v= \int_{\R^n}\int_{\R^n}e^{i(x-y)\xi}(\px^{r}-\langle y \rangle^{r})\varphi(\xi)q(\xi)v(y)dy \db \xi.$$
By the homogeneity properties of $q(\xi),$ arguing as in the proof of Lemma \ref{homodist}, we have that the kernel $K(x,y)$ of the operator above satisfies the following estimates:
$$|K(x,y)|\leq C \langle x-y \rangle^{-n-\mu+r}$$
and the same estimates hold for all the derivatives. In particular, if $\mu-r>-n/2$ then by Young inequality, the operator $[\px^r, H_{\varphi, q}]$ maps continuously $L^1(\R^n)$ into $L^2(\R^n)$, whereas if $\mu-r>0$, it is bounded on $L^2(\R^n)$. Similarly one can treat the derivatives and obtain Sobolev continuity and the estimates \eqref{sob1} and \eqref{sob2}. The lemma is then proved. \end{proof}

\end{section}


\begin{section}{Proof of the main result}
In this section we prove Theorem \ref{thm1}. We can assume without loss of generality that $F(u)=u^k$ for some integer $k \geq 2$ and that $p(\xi)$ is of the form \eqref{linop1} with $h=1, m=m_1=M$, i.e. $p(\xi)=p_0+p_m(\xi)$ with $p_m(\xi)$ non-polynomial positively homogeneous  function of order $m$ with $[m]>n/2.$ The extension to the general case is obvious. We first give a preliminary result.
\begin{lemma}\label{moreregdec}
Under the assumptions of Theorem \ref{thm1} we have $ u \in H^{s+1,1}(\R^n).$
\end{lemma}

\begin{proof}
We first prove that $u \in H^{s+1}(\R^n)$ that is $D_j u \in H^s(\R^n)$ for every $j \in \{1,\ldots, n\}.$ Differentiating \eqref{fequation} we obtain  $$P(D_ju)=D_jf+D_ju^k. $$ The assumption \eqref{ell} and the condition $M \geq 1$ imply that $P$ is invertible with symbol $1/p(\xi)$, and the operator $P^{-1}\circ D_j$ is bounded on $H^s(\R^n)$. Then we have $$D_ju=P^{-1}(D_jf)+P^{-1}(D_ju^k)$$ and since $u^k \in H^s(\R^n)$ by Schauder's estimates, we obtain $$\|D_ju \|_{s} \leq C_s (\|f\|_{s}+\|u\|_{s}^k)<\infty.$$ Starting from the assumption $\px^{\veps_o}u \in H^s(\R^n)$, we now prove by a bootstrap argument that $ u \in H^{s,1}(\R^n),$ that is $\px u \in H^s(\R^n)$. First, let $\ve_1< \min\{\ve_o, 1-\ve_o\}$, so that $\ve_o+\ve_1<1.$ Multiplying both sides of \eqref{fequation} by $\px^{\ve_o+\ve_1}$ and introducing commutators we have
$$P(\px^{\ve_o+\ve_1}u)=[P, \px^{\ve_o+\ve_1}]u + \px^{\ve_o+\ve_1}f+ \px^{\ve_o+\ve_1}u^k$$ and then 
\begin{equation}\label{boot}\px^{\ve_o+\ve_1}u= P^{-1}[P,\px^{\ve_o+\ve_1}]u +P^{-1}( \px^{\ve_o+\ve_1}f)+P^{-1}( \px^{\ve_o+\ve_1}u^k). \end{equation}
Now we write explicitly the commutator $$[P,\px^{\ve_o+\ve_1}]u = \iint e^{i( x-y) \xi}(\langle y \rangle^{\ve_o+\ve_1}-\px^{\ve_o+\ve_1})p_m(\xi)u(y)dy \dslash \xi.$$
Let $\varphi \in C_o^{\infty}(\R^n)$ such that $\varphi(\xi)=1$ for $|\xi|\leq 1$. Then we can decompose the commutator as follows:
$$[P,\px^{\ve_o+\ve_1}]u= Q_1u(x)+Q_2u(x),$$ where $$Q_1u(x)= \iint e^{i( x-y) \xi}(\langle y \rangle^{\ve_o+\ve_1}-\px^{\ve_o+\ve_1})\varphi(\xi)p_m(\xi)u(y)dy \dslash \xi,$$ and $$Q_2u(x)= \iint e^{i( x-y) \xi}(\langle y \rangle^{\ve_o+\ve_1}-\px^{\ve_o+\ve_1})(1-\varphi(\xi))p_m(\xi)u(y)dy \dslash \xi.$$ By Lemma \ref{fractional} with $q(\xi)=p_m(\xi), \mu=m$ and $r=\ve_o+\ve_1$, since $0<\ve_o+\ve_1<1$ and the condition $[m]>n/2$ implies $m\geq1,$ we have  $m-\ve_o-\ve_1>0$, then $Q_1$ is bounded on $H^s(\R^n)$, then the same is true for $P^{-1}\circ Q_1.$ On the other hand, $Q_2$ is an operator with smooth amplitude of order $m$, then by the classical theory, see \cite{Co, Ho}, we have that $P^{-1} \circ Q_2$ is bounded on $H^s(\R^n)$. In conclusion, we have that
$$\|P^{-1}[P, \px^{\ve_o+\ve_1}]u\|_{s} \leq C_s \| u\|_{s}< \infty.$$ Moreover, by Schauder's lemma we have, since $\ve_1<\ve_o$:
$$\|P^{-1}(  \px^{\ve_o+\ve_1}u^k)\|_{s} \leq C_s\|\px^{\ve_o+\ve_1}u^k\|_{s} \leq C'_s\|\px^{\veps_o}u\|^2_{s}\cdot \|u\|_{s}^{k-2} <\infty.$$
Hence $$\|\px^{\ve_o+\ve_1}u\|_{s} \leq C_s (\|\px^{\ve_o+\ve_1}f\|_{s} + \|u\|_{s}+\|\px^{\ve_o}u\|^2_{s}\cdot \|u\|_{s}^{k-2})<\infty.$$
Then $\px^{\ve_o+\ve_1}u \in H^s(\R^n)$. Possibly iterating this argument a finite number of times we obtain $\px^{\ve} u \in H^s(\R^n)$ for every $\ve \in (0,1).$ To obtain that $u \in H^{s, 1}(\R^n)$ we need a further step. Of course, it is sufficient to show that $x_h u \in H^s(\R^n)$ for any $h=1,\ldots, n.$ Arguing as for \eqref{boot} we have: $$
\| x_hu\|_s\leq C_s( \|P^{-1}[P,x_h]u\|_s +\|P^{-1}( x_h f)\|_s+\|P^{-1}( x_h u^k)\|_s).$$
Now, $P^{-1}[P,x_h]$ is the Fourier multiplier with symbol $\frac{(D_{\xi_h}p)(\xi)}{p(\xi)}$ which is bounded on $H^s(\R^n)$. Moreover,
$$\|P^{-1}( x_h u^k)\|_s \leq C_s \| x_h u^k\|_s \leq C'_s \|\px u^k \|_s \leq C''_s \|\px^{1/2}u \|^2_s \cdot \|u\|_s^{k-2}<\infty$$
by the previous step. Then we obtain that $x_h u \in H^s(\R^n), h=1,\ldots,n,$ i.e. $ u \in H^{s,1}(\R^n).$
Finally we prove that $x_h D_j u \in H^s(\R^n)$ for every $h,j \in \{1,\ldots,n \},$ that is $u \in H^{s+1,1}(\R^n).$ Starting from \eqref{fequation} and arguing as before we get
$$x_h D_j u=P^{-1}(x_h D_j f)+P^{-1}(x_h D_j u^k)+P^{-1}[P, x_h]D_j u.$$ Clearly  we have $\|  P^{-1}(x_h D_j f)\|_s < \infty.$ Moreover, 
\begin{multline*} \| P^{-1}(x_h D_j u^k)\|_s \leq C_s(\|P^{-1}\circ D_j(x_h u^k)\|_s + \|P^{-1}[x_h,D_j]u^k\|_s ) \\ \leq C'_s(\|\px u^k)\|_s + \|u^k\|_s)\leq C''_s\|\px u\|_s \cdot \|u\|_s^{k-1}<\infty. \end{multline*} Concerning the commutator we can repeat readily the argument used before and obtain that
$$\| P^{-1}[P, x_h]D_j u\|_s \leq C_s \|D_ju\|_s <\infty.$$ The lemma is then proved.
\end{proof}

\textit{Proof of Theorem \ref{thm1}.} We divide the proof in two steps. \\

\noindent
\textbf{First step.} Let us set 
\begin{equation}\label{kr}
k_{cr}= \max \{j \in \N : j < m+n/2\} = \begin{cases} [m+n/2] \quad \textit{if} \quad m+n/2 \notin \N \\ m+n/2-1 \quad \textit{if} \quad m+n/2 \in \N \end{cases}.\end{equation}
We first prove that $u \in C^{\infty}(\R^n)$ and $D^{\alpha}u \in H^{s,|\alpha|+k_{cr}}(\R^n)$ for every $\alpha \in \N^n$. This is equivalent to show that for every fixed $\alpha,\beta, \rho \in \N^n,$ with $|\beta|\leq |\alpha|$ and $|\rho|\leq k_{cr},$ we have $x^{\rho+\beta}D^{\alpha}u \in H^s(\R^n).$ This will be proved by induction on $|\rho+\alpha|.$ For $|\rho+\alpha|=1,$ the assertion is given by Lemma \ref{moreregdec}. Assume now that $x^{\rho+\beta}D^{\alpha}u \in H^{s}(\R^n)$ for $|\rho|\leq k_{cr}$, $|\beta|\leq |\alpha|$ and $|\rho+\alpha|\leq N$ for some positive integer $N$ and let us prove the same for $|\rho+\alpha|=N+1.$ We first apply $x^{\beta}D^{\alpha}$ to both sides of \eqref{fequation} and introduce commutators. We obtain
$$P(x^{\beta}D^{\alpha}u)=x^{\beta}D^{\alpha}f+x^{\beta}D^{\alpha}u^k -[ x^{\beta}D^{\alpha}, P]u.$$
By Proposition \ref{commid} we get
\begin{multline}
P(x^{\beta}D^{\alpha}u)=x^{\beta}D^{\alpha}f+x^{\beta}D^{\alpha}u^k \\ - \sum_{0\neq \gamma \leq \beta}\sum_{\stackrel{\ta, \tb}{|\tb|\leq|\ta|<|\alpha|}}C_{\alpha \beta \gamma \ta \tb \tg}D^{\tg}\circ (D_{\xi}^{\gamma}p_m)(D)(x^{\tb}D^{\ta}u),
\label{prima} \end{multline} 
where $|\tg|=|\gamma|.$
We now multiply both sides of \eqref{prima} by $x^{\rho}$ and write $$P(x^{\rho+\beta}D^{\alpha}u)=x^{\rho}P(x^{\beta}D^{\alpha}u)+[P,x^{\rho}](x^{\beta}D^{\alpha}u).$$ We have, by Proposition \ref{commid2}:
\begin{multline}
\label{seconda}
P(x^{\rho+\beta}D^{\alpha}u)= x^{\rho+\beta}D^{\alpha}f + x^{\rho+\beta}D^{\alpha}u^k  
- \hskip-5pt \sum_{0\neq \sigma \leq \rho}\hskip-3pt \binom{\rho}{\sigma}(-1)^{|\sigma|} (D_{\xi}^{\sigma}p_m)(D)(x^{\rho-\sigma+\beta}D^{\alpha}u) \\
\hskip-3pt - \sum_{0\neq \gamma \leq \beta}\sum_{\stackrel{\ta, \tb}{|\tb|\leq|\ta|<|\alpha|}}C_{\alpha \beta \gamma \ta \tb \tg}x^{\rho} D^{\tg}\circ (D_{\xi}^{\gamma}p_m)(D)(x^{\tb}D^{\ta}u).\end{multline}
Applying again Proposition \ref{commid2} with $q(\xi)= \xi^{\tg}D^{\gamma}_{\xi}p_m(\xi)$ and re-setting the sums, we obtain for new constants $C_{\alpha \beta \gamma \ta \tb \tg \rho \sigma}$
\begin{multline}
\label{terza}
P(x^{\rho}x^{\beta}D^{\alpha}u)= x^{\rho+\beta}D^{\alpha}f + x^{\rho+\beta}D^{\alpha}u^k \\ 
+ \sum_{\gamma \leq \beta, \sigma \leq \rho,|\sigma|\leq [m]}\sum_{\stackrel{|\tb|\leq |\ta|\leq|\alpha|}{|\alpha|-|\ta|+|\sigma|>0}}C_{\alpha \beta \gamma \ta \tb \tg \rho \sigma}p_m^{\gamma,\tg,\sigma}(D)(x^{\rho-\sigma+\tb}D^{\ta}u) \\  
+ \sum_{\gamma \leq \beta, \sigma \leq \rho,|\sigma|>m}\sum_{\stackrel{|\tb|\leq |\ta|\leq |\alpha|}{|\alpha|-|\ta|+|\sigma|>0}}C_{\alpha \beta \gamma \ta \tb \tg \rho \sigma}p_m^{\gamma,\tg,\sigma}(D)(x^{\rho-\sigma+\tb}D^{\ta}u),\end{multline}
where $p_m^{\gamma,\tg,\sigma}(D)$ is the Fourier multiplier with symbol $p_m^{\gamma, \tg, \sigma}(\xi)=D_{\xi}^{\sigma}(\xi^{\tg}D_{\xi}^{\gamma}p_m(\xi)).$ Notice that if $|\sigma|\leq [m],$ then $D_{\xi}^{\sigma}(\xi^{\tg}D_{\xi}^{\gamma}p_m(\xi))$ is well defined and locally bounded on $\R^n$, see Lemma \ref{homo}. If $|\sigma|>m$, then $m-|\sigma| \geq m-k_{cr}>-n/2$, then in particular $p_m^{\gamma,\tg,\sigma}(\xi) \in L^1_{\textrm{loc}}(\R^n)$ and defines a homogeneous distribution of order $m-|\sigma|.$ Let now $\varphi \in C^{\infty}_o(\R^n)$ with $\varphi(\xi)=1$ for $|\xi| \leq 1$. For $|\sigma|>m$ we can write $p_m^{\gamma, \tg, \sigma}(\xi)=p_{m,1}^{\gamma, \tg, \sigma}(\xi)+p_{m,2}^{\gamma, \tg, \sigma}(\xi)$, where $p_{m,1}^{\gamma, \tg, \sigma}(\xi)=(1-\varphi(\xi))p_m^{\gamma, \tg, \sigma}(\xi)$ and $p_{m,2}^{\gamma, \tg, \sigma}(\xi)= \varphi(\xi)p_m^{\gamma, \tg, \sigma}(\xi).$
Then we can invert $P$ and take Sobolev norms. We get
\begin{multline}
\label{quarta}
\|x^{\rho+\beta}D^{\alpha}u\|_{s} \leq \|P^{-1}( x^{\rho+\beta}D^{\alpha}f)\|_{s}+\|P^{-1}(x^{\rho+\beta}D^{\alpha}u^k)\|_{s}\\
\sum_{\gamma \leq \beta, \sigma \leq \rho,|\sigma|\leq [m]}\sum_{\stackrel{|\tb|\leq |\ta|\leq|\alpha|}{|\alpha|-|\ta|+|\sigma|>0}} |C_{\alpha\beta\gamma\ta\tb\tg\rho\sigma}|\cdot \|P^{-1}\circ p_m^{\gamma, \tg, \sigma}(D)(x^{\rho-\sigma+ \tb}D^{\ta}u)\|_{s} \\
+\sum_{\gamma \leq \beta, \sigma \leq \rho,|\sigma|>m}\sum_{\stackrel{|\tb|\leq |\ta|\leq|\alpha|}{|\alpha|-|\ta|+|\sigma|>0}} |C_{\alpha\beta\gamma\ta\tb\tg\rho\sigma}|\cdot \|P^{-1}\circ p_{m,1}^{\gamma, \tg, \sigma}(D)(x^{\rho-\sigma+ \tb}D^{\ta}u)\|_{s} \\+\sum_{\gamma \leq \beta,\sigma \leq \rho,|\sigma|>m}\sum_{\stackrel{|\tb|\leq |\ta|\leq|\alpha|}{|\alpha|-|\ta|+|\sigma|>0}}|C_{\alpha\beta\gamma\ta \tb\tg\rho\sigma}|\cdot \|P^{-1}\circ p_{m,2}^{\gamma,\tg,\sigma}(D)(x^{\rho-\sigma+\tb}D^{\ta}u)\|_{s},
\end{multline}
where $p_{m,j}^{\gamma, \tg, \sigma}(D),$ $j=1,2$ denote the operators associated to the symbols $p_{m,j}^{\gamma, \tg, \sigma}(\xi),$ $j=1,2.$ 
We want to estimate the five terms in the right-hand side of \eqref{quarta}. The first is finite by assumption. Concerning the nonlinear term,
if $\rho=\beta=0$, by the boundedness of $P^{-1} \circ D_{j}, j=1,\ldots,n,$ using Leibniz formula and Schauder's estimates, we get:
$$\|P^{-1}D^{\alpha}u^k\|_{s} \leq C_s \|D^{\alpha-e_j}u^k\|_{s}\leq C_{s\alpha}\|u\|_{s+|\alpha|-1}^k <\infty$$ by the inductive assumption. 
If $\rho+\beta \neq 0,$ we can write:
$$x^{\rho+\beta}D^{\alpha}u^{k}= k x^{\rho+\beta}u^{k-1}D^{\alpha}u + x^{\rho+\beta}\sum_{\stackrel{\alpha_{1}+\ldots+\alpha_{k}=\alpha}{|\alpha_{j}|<|\alpha|\forall j}} \frac{\alpha!}{\alpha_{1}! \ldots \alpha_{k}!} D^{\alpha_{1}}u \cdot \ldots \cdot D^{\alpha_{k}}u.$$ Moreover, since $|\beta|\leq |\alpha|,$ we can write $\beta=\beta_{1}+\ldots+\beta_{k}$ for some $\beta_{j}$ satisfying $|\beta_{j}|\leq |\alpha_{j}|, j=1,\ldots, k.$ Then we have, for some $\ell \in \{1,\ldots, n\}:$
\begin{multline*}\|x^{\rho+\beta}D^{\alpha}u^{k}\|_{s} \leq C_{s,\alpha}(\|x^{\rho+\beta-e_{\ell}}D^{\alpha}u\|_{s} \cdot \|x_{\ell}u\|_{s}\cdot \|u\|_{s}^{k-2}  \\
 + \sum_{\stackrel{\alpha_{1}+\ldots+\alpha_{k}=\alpha}{|\alpha_{j}|<|\alpha|\forall j}}\|x^{\rho+\beta_{1}}D^{\alpha_{1}}u\|_{s}\cdot \prod_{j=2}^{k}\|x^{\beta_{j}}D^{\alpha_{j}}u\|_{s}) <\infty
\end{multline*}
by the inductive assumption. The third and the fourth term in the right-hand side of \eqref{quarta} can be easily estimated inductively observing that by Lemma \ref{homo}, the operators $P^{-1} \circ p_m^{\gamma,\tg,\sigma}(D)$ with $|\sigma| \leq [m]$ and $P^{-1}\circ p_{m,1}^{\gamma,\tg,\sigma}(D)$ are both bounded on $H^s(\R^n)$ and that $|\tb|\leq |\ta|$ and  $|\rho-\sigma+\ta|<|\rho+\alpha|$ since $|\alpha|-|\ta|+|\sigma|> 0.$ Concerning the last term, the estimate is more delicate since we have to deal with singular operators. Nevertheless, we can apply Lemma \ref{homodist} with $q(\xi)=p_m^{\gamma,\tg,\sigma}(\xi)$ and $\mu=m-|\sigma|\geq m-k_{cr}>-n/2$. We obtain
$$\|P^{-1}\circ p^{\gamma, \tg,\sigma}_{m,2}(D)(x^{\rho-\sigma+ \tb}D^{\ta}u )\|_{s} \leq C_s \|x^{\rho-\sigma+ \tb}D^{\ta}u \|_{H^s_1}.$$
 Moreover, 
\begin{equation}\label{matter}\|x^{\rho-\sigma+ \tb}D^{\ta}u \|_{H^s_1} \leq C_s \|\px^{|\rho|-1+ |\tb|}D^{\ta}u \|_{s}.\end{equation} As a matter of fact, we have, for some $k \in \{1,\ldots, n\}:$ 
\begin{eqnarray}\|x^{\rho-\sigma+ \tb}D^{\ta}u \|_{L^1} &\leq& C \| \px^{-|\sigma|+1}\px^{|\rho|-1+ |\tb|}D^{\ta}u\|_{L^1}\nonumber \\ &\leq& C \|\px^{-|\sigma|+1}\|_{L^2} \cdot \|\px^{|\rho|-1+ |\tb|}D^{\ta}u\|_{L^2} \nonumber \\
&\leq& C' \|\px^{|\rho|-1+ |\tb|}D^{\ta}u\|_{L^2}, \label{spinosa} \end{eqnarray}
by H\"older inequality, since the condition $|\sigma|>m$ implies $|\sigma|\geq [m]+1>n/2+1$ and this gives $\px^{-|\sigma|+1} \in L^2(\R^n).$ Similar estimates can be proved for the derivatives and give \eqref{matter}.  In conclusion, we obtain 
$$\|P^{-1}\circ p_{m,2}^{\gamma,\tg,\sigma}(D)(x^{\rho-\sigma+ \tb}D^{\ta}u)\|_{s} \leq C_s \|\px^{|\rho|-1 +|\tb|}D^{\ta}u \|_{s}<\infty $$ by the inductive assumption.\\

\noindent \textbf{Second step.} Let now $\tau$ be the fractional part, i.e. $0<\tau<1, k_{cr}+\tau<m+n/2.$ To conclude the proof we need to prove that $\px^{\tau}x^{\rho+\beta}D^{\alpha}u \in H^s(\R^n)$ for every $\rho, \alpha, \beta \in \N^n$ with $|\beta|\leq |\alpha|$ and $|\rho|\leq k_{cr}.$ Starting from the identity \eqref{terza}, multiplying both sides by $\px^{\tau}$ and introducing commutators, we obtain:
\begin{multline} \label{quinta}
P(\px^{\tau}x^{\rho+\beta} D^{\alpha}u) = [P, \px^{\tau}](x^{\rho+\beta} D^{\alpha}u) +\px^{\tau}P(x^{\rho+\beta}D^{\alpha }u)\\
=[P, \px^{\tau}](x^{\rho+\beta} D^{\alpha}u) +\px^{\tau}x^{\rho+\beta}D^{\alpha }f+\px^{\tau}x^{\rho+\beta}D^{\alpha }u^k \\+ \hskip-4pt \sum_{\gamma \leq \beta, \sigma \leq \rho,|\sigma|\leq [m]}\sum_{\stackrel{|\tb|\leq |\ta|\leq |\alpha|}{|\alpha|-|\ta|+|\sigma|>0}}C_{\alpha \beta \gamma \ta \tb \tg \rho \sigma}p_m^{\gamma,\tg,\sigma}(D)(\px^{\tau}x^{\rho-\sigma+ \tb}D^{\ta}u) \\ +  \hskip-4pt \sum_{\gamma \leq \beta, \sigma \leq \rho,|\sigma|\leq [m]}\sum_{\stackrel{|\tb|\leq |\ta|\leq |\alpha|}{|\alpha|-|\ta|+|\sigma|>0}}C_{\alpha \beta \gamma \ta \tb \tg \rho \sigma}[\px^{\tau},p_m^{\gamma,\tg,\sigma}(D)](x^{\rho-\sigma+ \tb}D^{\ta}u) \\ +\sum_{\gamma \leq \beta, \sigma \leq \rho,|\sigma|> [m]}\sum_{\stackrel{|\tb|\leq |\ta|\leq |\alpha|}{|\alpha|-|\ta|+|\sigma|>0}}C_{\alpha \beta \gamma \ta \tb \tg \rho \sigma}p_{m,1}^{\gamma,\tg,\sigma}(D)(\px^{\tau} x^{\rho-\sigma+ \tb}D^{\ta}u) \\+\sum_{\gamma \leq \beta, \sigma \leq \rho,|\sigma|> [m]}\sum_{\stackrel{|\tb|\leq |\ta|\leq |\alpha|}{|\alpha|-|\ta|+|\sigma|>0}}C_{\alpha \beta \gamma \ta \tb \tg \rho \sigma}[\px^{\tau},p_{m,1}^{\gamma,\tg,\sigma}(D)](x^{\rho-\sigma+ \tb}D^{\ta}u) \\ + \sum_{\gamma \leq \beta, \sigma \leq \rho,|\sigma|> [m]}\sum_{\stackrel{|\tb|\leq |\ta|\leq |\alpha|}{|\alpha|-|\ta|+|\sigma|>0}}C_{\alpha \beta \gamma \ta \tb \tg \rho \sigma}\px^{\tau}p_{m,2}^{\gamma,\tg,\sigma}(D)(x^{\rho-\sigma+ \tb}D^{\ta}u).
 \end{multline}
At this point we can apply $P^{-1}$ to both sides of \eqref{quinta} and take Sobolev norms. We already know that $P^{-1}$ and $P^{-1} \circ p_m^{\gamma, \tg,\sigma}(D)$ for $|\sigma| \leq [m]$ and $P^{-1}[P, \px^{\tau}]$ are bounded on $H^s(\R^n)$. Moreover, we recall that $p_{m,1}^{\gamma,\tg,\sigma}(D)$ is a Fourier multiplier with smooth symbol of negative order, then it is bounded on $H^s(\R^n)$. For the same reason, since $\tau <1$, we have that $[\px^{\tau},p_{m,1}^{\gamma,\tg,\sigma}(D)]$ is a pseudodifferential operator with smooth and bounded symbol, then it is also bounded on $H^s(\R^n)$.  We obtain
\begin{multline}
\|\px^{\tau}x^{\rho+\beta} D^{\alpha}u\|_{s} \leq C_s(\|x^{\rho+\beta}D^{\alpha }u\|_s + \|\px^{\tau}x^{\rho+\beta}D^{\alpha }f\|_{s} + \|\px^{\tau}x^{\rho+\beta}D^{\alpha }u^k\|_{s})\\
+\sum_{\gamma \leq \beta, \sigma \leq \rho,|\sigma|\leq [m]}\sum_{\stackrel{|\tb|\leq |\ta|\leq |\alpha|}{|\alpha|-|\ta|+|\sigma|>0}}C_{s\alpha \beta \gamma \ta \tb \tg \rho \sigma}\cdot \|\px^{\tau}x^{\rho-\sigma+ \tb}D^{\ta}u\|_{s} \\
+\sum_{\gamma \leq \beta, \sigma \leq \rho,|\sigma|\leq [m]}\sum_{\stackrel{|\tb|\leq |\ta|\leq |\alpha|}{|\alpha|-|\ta|+|\sigma|>0}}C_{s\alpha \beta \gamma \ta \tb \tg \rho \sigma} \cdot \|[\px^{\tau},p_m^{\gamma,\tg,\sigma}(D)]x^{\rho-\sigma+ \tb}D^{\ta}u\|_{s}\\
+\sum_{\stackrel{\gamma \leq \beta, \sigma \leq \rho}{|\sigma|>m}}\sum_{\stackrel{|\tb|\leq |\ta|\leq|\alpha|}{|\alpha|-|\ta|+|\sigma|>0}}C_{s\alpha \beta \gamma \ta \tb \tg  \rho \sigma } \|\px^{\tau}x^{\rho-\sigma+ \tb}D^{\ta}u\|_{s} \\
\sum_{\stackrel{\gamma \leq \beta, \sigma \leq \rho}{ |\sigma|> m}}\sum_{\stackrel{|\tb|\leq |\ta|\leq|\alpha|}{|\alpha|-|\ta|+|\sigma|>0}}C_{s\alpha \beta \gamma \ta \tb\tg  \rho \sigma} \|x^{\rho-\sigma+ \tb}D^{\ta}u\|_{s} \\
+\sum_{\stackrel{\gamma \leq \beta, \sigma \leq \rho}{ |\sigma|> m}}\sum_{\stackrel{|\tb|\leq |\ta|\leq|\alpha|}{|\alpha|-|\ta|+|\sigma|>0}}C_{s\alpha \beta \gamma \ta \tb \tg \rho \sigma } \|\px^{\tau}p_{m,2}^{\gamma,\tg,\sigma}(D)(x^{\rho-\sigma+ \tb}D^{\ta}u)\|_{s},
\label{sesta}\end{multline}
where $C_{s\alpha \beta \gamma \ta \tb \tg \rho \sigma }$ are positive constants.
Let us now estimate the terms in the right-hand side of \eqref{sesta}. The first is finite by the previous step of the proof, the second by assumption. Concerning the nonlinear term, we can write as before
\begin{multline*}
\px^{\tau}x^{\rho+\beta}D^{\alpha }u^k= k x^{\rho+\beta}D^{\alpha }u \cdot \px^{\tau}u \cdot u^{k-2} \\
+ \sum_{\stackrel{\alpha_1+\ldots+\alpha_k=\alpha}{|\alpha_j|<|\alpha|}}\frac{\alpha!}{\alpha_1!\ldots \alpha_k!}x^{\rho+\beta_1}D^{\alpha_1}u \cdot \px^{\tau}x^{\beta_2}D^{\alpha_2}u \cdot \prod_{j=3}^{k}x^{\beta_j}D^{\alpha_j}u. \end{multline*}
where $|\beta_j|\leq|\alpha_j|, j=1,\ldots, k$ and the last product does not appear if $k=2.$
Then we have the following estimate:
\begin{multline*}\|\px^{\tau}x^{\rho+\beta}D^{\alpha }u^k\|_{s} \leq C_s \|x^{\rho+\beta}D^{\alpha }u\|_{s}\cdot \|\px^{\tau}u\|_{s} \cdot \|u\|_{s}^{k-2} \\+  \sum_{\stackrel{\alpha_1+\ldots+\alpha_k=\alpha}{|\alpha_j|<|\alpha|}}\frac{\alpha!}{\alpha_1!\ldots \alpha_k!}
\| x^{\rho+\beta_1}D^{\alpha_1}u\|_{s} \cdot \|\px^{\tau}x^{\beta_2}D^{\alpha_2}u\|_{s}\cdot \prod_{j=3}^{k}\|x^{\beta_j}D^{\alpha_j}u\|_{s} \\
\leq C_s \|u\|_{s+|\alpha|,k_{cr}+|\alpha|} \cdot \|\px^{\tau} u\|_{s}\cdot \|u\|_{s}^{k-2} + C_{s\alpha}\|u\|^k_{s+|\alpha|, k_{cr}+|\alpha|} <\infty, \end{multline*} since $\tau<1 \leq k_{cr}.$
To estimate the fourth term, we observe that, since $|\alpha|-|\ta|+|\sigma|>0$, then, if $\sigma \neq 0$, we have $$\|\px^{\tau}x^{\rho-\sigma+ \tb}D^{\ta}u\|_{s}\leq C_s \|\px^{|\rho|+|\tb|}D^{\ta}u\|_s <\infty.$$ If $\sigma =0$, then $|\alpha|-|\ta|>0$ so that $|\tb|+1 \leq |\ta|+1\leq |\alpha|$. Hence 
$$ \|\px^{\tau}x^{\rho+ \tb}D^{\ta}u\|_{s}\leq C_s \|\px^{|\rho|+|\tb|+1}D^{\ta}u \|_s \leq C'_s \|u \|_{s+|\alpha|,k_{cr}+|\alpha|}<\infty$$ by the previous step. The fifth term is more delicate to estimate. After cutting-off the amplitude of the commutator we can apply Lemma \ref{fractional} with $r=\tau, q(\xi)= p_m^{\gamma,\tg,\sigma}(\xi), \mu=m-|\sigma|$ and since $m-|\sigma|-\tau>-n/2$, the operator $[\px^{\tau},p_m^{\gamma,\tg,\sigma}(D)]$ can be written as the sum of a bounded operator on $H^s(\R^n)$ and a continuous operator $H_1^s(\R^n) \to H^s(\R^n).$ Hence we have:
\begin{multline*}\| [\px^{\tau},p_m^{\gamma,\tg,\sigma}(D)] x^{\rho-\sigma+ \tb}D^{\ta}u\|_{s} \leq C_s (\| x^{\rho-\sigma+ \tb}D^{\ta}u\|_{s}+\|x^{\rho-\sigma+ \tb}D^{\ta}u\|_{H^s_1}) \\ C'_s \|\px^{|\rho|+|\tb|}D^{\ta}u\|_s <\infty \end{multline*}
by H\"older inequality, since $|\sigma|>n/2.$
The sixth and the seventh term in the right-hand side of \eqref{sesta} are obviously finite. Concerning
the last term, we can write  \begin{eqnarray*}\px^{\tau}p_{m,2}^{\gamma,\tg,\sigma}(D)(x^{\rho-\sigma+ \tb}D^{\ta}u) &=& p_{m,2}^{\gamma,\tg,\sigma}(D)(\px^{\tau}x^{\rho-\sigma+ \tb}D^{\ta}u) \\ &&- [p_{m,2}^{\gamma,\tg,\sigma}(D), \px^{\tau}](x^{\rho-\sigma+ \tb}D^{\ta}u)
\end{eqnarray*} and apply Lemmas \ref{homodist} and \ref{fractional} with $q(\xi)=p_m^{\gamma, \tg, \sigma}(\xi), r=\tau, \mu=m-|\sigma|$. We obtain:
\begin{multline*}\|\px^{\tau}p_{m,2}^{\gamma,\tg,\sigma}(D)(x^{\rho-\sigma+ \tb}D^{\ta}u)\|_{s}  \leq C_s\|\px^{\tau} x^{\rho-\sigma+\tb}D^{\ta}u\|_{H^s_1} \\
\leq C'_{s}\| \px^{\tau-|\sigma|}\px^{|\rho|+|\tb|}D^{\ta}u\|_{H^s_1}\leq C''_s \|\px^{|\rho|+|\tb|}D^{\ta}u\|_{s}<\infty. \end{multline*}
arguing as in the proof of \eqref{spinosa}.
The theorem is then proved.
\qed
\end{section}


\vspace{1cm}
\noindent
$^{\textrm{a}}$ Dipartimento di Matematica, Universit\`a di Torino, via Carlo Alberto 10, 10123 Torino, Italy (Corresponding author) \\
{\small\bf email: marco.cappiello@unito.it}\\

\noindent
$^{\textrm{b}}$ Dipartimento di Matematica e Informatica, Universit\`a di Cagliari, via Ospedale 72, 09124 Cagliari, Italy \\
{\small\bf email: todor@unica.it}\\ 

\noindent
$^{\textrm{c}}$ Dipartimento di Matematica, Universit\`a di Torino, via Carlo Alberto 10, 10123 Torino, Italy  \\
{\small\bf email: luigi.rodino@unito.it} 

\end{document}